\documentclass[11pt,reqno]{amsart}
\usepackage{color}
\usepackage{amsmath,amssymb,amsfonts,amsthm,bm}
\usepackage{mathrsfs}
\usepackage{graphicx}
\usepackage{enumerate}
\usepackage[numbers, sort&compress]{natbib}
\usepackage[shortlabels]{enumitem}
 \usepackage{a4wide}

\newtheorem{theorem}{Theorem}[section]

\theoremstyle{assumption}
\newtheorem{assumption}[theorem]{Assumption}
\theoremstyle{definition}
\newtheorem{definition}[theorem]{Definition}
\theoremstyle{remark}
\newtheorem{remark}[theorem]{Remark}
\numberwithin{equation}{section}

\newcommand{\eps}{\varepsilon} 
\newcommand{\norm}[1]{\Vert#1\Vert}
\newcommand{\abs}[1]{\left\vert#1\right\vert}

\newcommand{\inner}[1]{\left(#1\right)}
\newcommand{\comi}[1]{\left<#1\right>}

\newcommand{\normm}[1]{{ \vert\kern-0.25ex \vert\kern-0.25ex \vert #1 
		\vert\kern-0.25ex \vert\kern-0.25ex \vert}}

\makeatletter

\@namedef{subjclassname@2020}{%
	\textup{2020} Mathematics Subject Classification}

\def\@startsection#1#2#3#4#5#6{%
	\if@noskipsec \leavevmode \fi
	\par \@tempskipa #4\relax
	\@afterindentfalse
	\ifdim \@tempskipa <\z@ \@tempskipa -\@tempskipa \@afterindentfalse\fi
	\if@nobreak \everypar{}\else
	\addpenalty\@secpenalty\addvspace\@tempskipa\fi
	\@ifstar{\@dblarg{\@sect{#1}{\@m}{#3}{#4}{#5}{#6}}}%
	{\@dblarg{\@sect{#1}{#2}{#3}{#4}{#5}{#6}}}%
}

\def\@settitle{%
	\bgroup
	\centering
	\vglue1cm
	\fontseries{b}\selectfont
	\uppercasenonmath\@title
	\@title
	\vskip20pt plus 6pt minus 8pt
	\egroup
}

\def\@setauthors{%
	\begingroup
	\trivlist
	\centering \bfseries
	\normalsize\@topsep30\p@\relax
	\advance\@topsep by -\baselineskip
	\item\relax
	\andify\authors
	{\rmfamily\authors}%
	\endtrivlist
	\endgroup
}

\def\@setaddresses{\par
	\nobreak \begingroup
	\normalsize
	\def\author##1{\nobreak\addvspace\bigskipamount}%
	\def\\{\unskip, \ignorespaces}%
	\interlinepenalty\@M
	\def\address##1##2{\begingroup
		\par\addvspace\bigskipamount\noindent
		\@ifnotempty{##1}{(\ignorespaces##1\unskip) }%
		{\ignorespaces##2}\par\endgroup}%
	\def\curraddr##1##2{\begingroup
		\@ifnotempty{##2}{\nobreak\indent{\itshape Current address}%
			\@ifnotempty{##1}{, \ignorespaces##1\unskip}\/:\space
			##2\par}\endgroup}%
	\def\email##1##2{\begingroup
		\@ifnotempty{##2}{\nobreak\noindent{\itshape E-mail address}%
			\@ifnotempty{##1}{, \ignorespaces##1\unskip}\/: 
			##2\par}\endgroup}%
	\def\urladdr##1##2{\begingroup
		\@ifnotempty{##2}{\nobreak\indent{\itshape URL}%
			\@ifnotempty{##1}{, \ignorespaces##1\unskip}\/:\space
			\ttfamily##2\par}\endgroup}%
	\addresses
	\endgroup
}

\renewcommand\section{\@startsection{section}{1}{\z@}%
	{27pt plus 6pt minus 8pt}{14pt plus 6pt minus 8pt}
	{\center\normalfont\large\bfseries}}

\makeatother

\begin{document}
	
\title[Well-posedness of the hyperbolic Prandtl equations]{Gevrey well-posedness of the hyperbolic Prandtl  equations}

\author[W.-X. Li and R. Xu]{Wei-Xi Li  \and Rui Xu}

	\date{}

	\address[W.-X. Li]{ 
		School of Mathematics and Statistics,  Wuhan University,  430072 Wuhan, China \\ 
 \& Hubei Key Laboratory of Computational Science,    Wuhan University,  430072 Wuhan, China
	}

	\email{
		wei-xi.li@whu.edu.cn}
		
\address[R. Xu]{ 
		School of Mathematics and Statistics,     Wuhan University,  430072 Wuhan, China
	}		
\email{xurui218@whu.edu.cn}



\maketitle

\begin{abstract}
 We study the 2D and 3D  Prandtl equations of degenerate hyperbolic type,   and  establish  without any structural assumption the Gevrey well-posedness with Gevrey index $\leq 2$. Compared with the classical parabolic Prandtl equations,   the loss of the derivatives, caused by the hyperbolic feature coupled with the degeneracy, can't be overcame by virtue of the  classical cancellation mechanism that developed for the parabolic counterpart. 
 Inspired by the abstract Cauchy-Kowalewski theorem and by virtue of the hyperbolic feature,    we give in this text    a straightforward proof, basing on an elementary $L^2$ energy estimate. In particular our argument does not involve the cancellation mechanism used efficiently for the classical Prandtl equations.
\end{abstract}

\section{Introduction and main result}

The classical  Prandtl equation was derived by L.Prandtl  in 1904 as a model to describe the behavior of the flow near the fluid boundary.  It is a degenerate parabolic type equation, losing one order tangential derivatives due to the absence of tangential diffusion.   The mathematical study on  the classical Prandtl boundary layer has a  long history with various approaches developed    and   so far  it is  well-explored in various function spaces,  see, e.g., \cite{MR3327535, MR3795028, MR3670620,  MR3983729, MR3458159, MR3600083, MR3925144,  MR1476316, MR2601044, MR3429469, MR2849481,  MR3461362, MR3284569,     MR3493958, MR4055987,  MR2020656, MR3710703,MR3464051} and the references therein.  Among those works,  there are   two   main settings. One is referred to the Sobolev well/ill-posedness that usually requires  Oleinik's monotonicity assumption (cf. \cite{MR3327535,MR3385340,MR1697762}), and  another is  referred to the Gevrey (including analyticity) well-posedness without any structural assumption.   Note that Oleinik's monotonicty condition is  crucial for the Sobolev well-posedness theory.  On one hand, the classical solution was established by Oleinik (see for instance \cite{MR1697762}) and two groups Alexandre-Wang-Xu-Yang \cite{MR3327535} and Masmoudi-Wong \cite{MR3385340} independently proved the existence and uniqueness theory in the Sobolev spaces by using energy method.   On the other hand,   without Oleinik's monotonicty, the Prandtl equantion may be ill-posed in Sobolev space  (Gevrey space more precisely) that was observed by G\'erard-Varet and Dormy \cite{MR2601044}.  Inspired by the ill-posedness result in \cite{MR2601044, MR3670620},  it is natural  to  expect the Gevery well-posedness whenever the initial data are in Gevery class  $\leq 2$  without any structural assumption.  This  has been confirmed recently  by Dietert and G\'erard-Varet \cite{MR3925144} and  Li-Masmoudi-Yang \cite{lmy}   for the 2D and  the 3D Prandtl equations, respectively,  after the earlier works of \cite{MR3795028, MR3429469,MR4055987} in Gevrey class and  Sammartino-Caflisch's work \cite{MR1617542} in  the analytic setting. Recently, similar problems have been explored   for  MHD boundary layer equations  that are of Prandtl type equations with the strong background magnetic fields; for instance   the stabilizing effect of the magnetic fields       has  been  justified recently by  Liu-Xie-Yang \cite{MR3882222}(see also  \cite{MR4342301, MR4102162} for the further development), and the ill-posed in Sobolev space (more precisely, in Gevrey spaces with index $>2$)  is establised by \cite{MR4102162, MR3864769} and the Gevrey well-posedness is proven by  \cite{MR4270479}.     The aforementioned works are mainly concerned with the local-in-time existence.  We refer to    Xin-Zhang's work  \cite{MR2020656} on global weak solutions under the monotonicity condition and the recent works of Paicu-Zhang \cite{MR4271962} and  Wang-Wang-Zhang \cite{2021arXiv210300681Y} on global small analytic and Gevrey solutions, respectively. The global analytic solutions to MHD boundary layer system were obtainded recently by Liu-Zhang \cite{ MR4213671} and Li-Xie \cite{MR4293727}.

 This work aims to study the well-posedness theory of the hyperbolic version of Prandtl equations.   The hyperbolic boundary layer equations can be  
 derived similarly as the classical Prandtl system,  just following the Prandtl's ansatz when analyzing  the  asymptotic expansion w.r.t. the viscosity  of  the  hyperbolic Navier-Stokes equations which are complemented with  the  no-slip boundary conditions. The hyperbolic version of Navier-Stokes equation  by adding a small hyperbolic perturbation to the classical one,  was initiated by   Cattaneo \cite{MR0032898}.     The classical Navier-Stokes equation has infinite propagation speed which is  non-physical,  and  in order to avoid the non-physical feature Cattaneo \cite{MR0032898} proposed the Cattaneo law instead of the classical Fourier law.  Up to now, there have been extensive  works on the hyperbolic Navier-Stokes equation; see for instance \cite{MR3942552,MR2045417,CHR,MR3085225,MR3085226,MR2404054} and the references therein.  Similar to the  parabolic Naver-Stokes equation, a boundary layer will appear when investigating the hyperbolic counterpart  in a bounded domain  complemented with   the no-slip boundary conditions, and  its governing equation can be derived as below (see Appendix \ref{secderivation} below), just  following the classical Prandtl's ansatz,
 \begin{equation}\label{mhdsys+++}
\left\{
\begin{aligned}
&\big(\partial_t^2   +\partial_t    +u\cdot\partial_x  +v\partial_{y} -\partial_{y}^2\big)u+\partial_xp=0, \quad (x, y)\in  \mathbb R^{n}_+ \\
&\partial_x\cdot u+\partial_{y}v=0, \quad (x, y)\in  \mathbb R^{n}_+ \\
&u|_{ y=0}=0,  \  v|_{ y=0}=0,\quad \lim_{y\rightarrow +\infty}u(t,x,y)=U(t,x), \quad  x\in\mathbb R^{n-1},\\
& u|_{ t=0}=  u_{0} , \quad \partial_t u|_{ t=0}= u_{1}, \quad (x, y)\in  \mathbb R^{n}_+, \\
\end{aligned}
\right.
\end{equation}
where the fluid domain $ \mathbb R^{n}_+=\big\{ (x,y)\in\mathbb R^n;\  x\in\mathbb R^{n-1},  y>0\big \}$ with  $n=2$ or $n=3$,   and  $u, v$   stand for the  
tangential and normal velocity, respectively,  and   $ U(t, x) , p(t, x) $ are  given functions from  the   outflow   that are linked by 
\begin{equation*}
\partial_t^2U+\partial_t U+(U\cdot\partial_x)U+\partial_xp=0.
\end{equation*}   
 For simplicity,  we may assume without loss of generality   that  $U\equiv  0$ in   \eqref{mhdsys+++},  and then the system \eqref{mhdsys+++}  is reduced to
 \begin{equation}\label{hp}
 \left\{
 \begin{aligned}
 &\big(\partial_t^2    +\partial_t    +u\cdot\partial_x +v\partial_{y} -\partial_{y}^2\big)u=0,\\
 &\partial_x\cdot u+\partial_{y}v=0,\\
 &u|_{ y=0}= 0, \quad v|_{ y=0}=0,\quad \lim_{y\rightarrow +\infty}u(t,x,y)= 0,\\
 & u|_{ t=0}=u_{0}, \quad \partial_t u|_{ t=0}=u_{1}.\\
 \end{aligned}
 \right.
 \end{equation}
This text is motivated by the recent work of   Aarach   \cite{2021arXiv211113052A}  and Paicu-Zhang\cite{2021arXiv211112836P}, where the authors investigated  the global well-posedness of  hydrostatic Navier-Stokes equations of hyperbolic version.  For  the classical hydrostatic Navier-Stokes or Euler   equations, it is far from well-explored, and there are only few works. We mention  here the very recent work \cite{MR4149066} of   G\'erard-Varet-Masmoudi-Vicol on the local Gevrey well-posedness of 2D parabolic hydrostatic Navier-Stokes equation under convex assumption. The convex assumption is removed  for the hyperbolic version of  hydrostatic Navier-Stokes equation by the recent  Paicu-Zhang's work \cite{2021arXiv211112836P} where they established the global well-posedness in  Gevrey class 2,  improving the earlier  Aarach's work   \cite{2021arXiv211113052A} in the analytic setting.  This shows the hyperbolic feature  may acts as  stabilizing factor for hydrostatic Navier-Stokes equation.  In this work we will justify the effect of the hyperbolic feature on the Prandtl equations that admit a degeneracy structure different from the hydrostatic Navier-Stokes equation.

 \smallskip
  {\bf Notation.}  Given  the domain $\Omega= \mathbb R_+^n$ with $n=2$ or $3$,   we will use  $\norm{\cdot}_{L^2}$ and $\inner{\cdot, \cdot}_{L^2}$ to denote the norm and inner product of  $L^2=L^2(\Omega)$   and use the notation   $\norm{\cdot}_{L_x^2}$ and $\inner{\cdot, \cdot}_{L_x^2}$  when the variable $x$ is specified. Similar notation  will be used for $L^\infty$.  Moreover, we use $L^p_x(L^q_y) = L^p (\mathbb T; L^q(\mathbb R_+))$ for the classical Sobolev space.
      
\begin{definition} 
\label{defgev}  Let $\ell>1/2 $ be a given number.   With each pair $(\rho,\sigma)$, $\rho>0$ and $\sigma\geq 1, $  the Gevrey  space $X_{\rho,\sigma}$   consists of all  smooth   functions
 $f(t,x)$  such that the Gevrey norm  $\abs{f}_{\rho,\sigma}<+\infty,$  where    $\abs{\cdot}_{\rho,\sigma}$ is defined as below. 
 If $f$ is a function of $x$ variable only  but independent of $t$, then
 \begin{eqnarray*}
\begin{aligned}
	 \abs{f}_{\rho,\sigma}= &\sup_{ \abs\alpha\geq 7} \frac{\rho^{\abs\alpha-7}}{[\inner{\abs\alpha-7}!]^{\sigma}} \Big(   \big\|\comi y^{\ell }  \partial_y\partial_x^\alpha      f \big\|_{L^2}+\abs\alpha\norm{\comi y^{\ell } \partial_x^\alpha      f }_{L^2}\Big)\\
	 &+\sup_{ \abs\alpha\leq 6}  \Big(   \big\|\comi y^{\ell }  \partial_y\partial_x^\alpha      f \big\|_{L^2} +\abs\alpha\norm{\comi y^{\ell } \partial_x^\alpha      f}_{L^2}\Big),
	 \end{aligned}
\end{eqnarray*}
and moreover for functions $f$ of $(t,x)$ variables we define  
\begin{eqnarray*}
\begin{aligned}
	 \abs{f(t)}_{\rho,\sigma}= &\sup_{ \abs\alpha\geq 7} \frac{\rho^{\abs\alpha-7}}{[\inner{\abs\alpha-7}!]^{\sigma}} \Big(  \big\|\comi y^{\ell}  \partial_t\partial_x^\alpha f(t)\big\|_{L^2}+ \big\|\comi y^{\ell }  \partial_y\partial_x^\alpha      f(t)\big\|_{L^2}+\abs\alpha\norm{\comi y^{\ell } \partial_x^\alpha      f(t)}_{L^2}\Big)\\
	 &+\sup_{ \abs\alpha\leq 6}  \Big(  \big\|\comi y^{\ell}  \partial_t\partial_x^\alpha f(t)\big\|_{L^2}+ \big\|\comi y^{\ell }  \partial_y\partial_x^\alpha      f(t)\big\|_{L^2}+\abs\alpha\norm{\comi y^{\ell } \partial_x^\alpha      f(t)}_{L^2}\Big),
	 \end{aligned}
\end{eqnarray*}
where and throughout the paper $\comi y=(1+\abs y^2)^{1/2}.$  We call $\sigma$ the Gevrey index.  
\end{definition} 

Note if $f\in X_{\rho,\sigma}$, then $f$ is of Gevrey class and Sobolev, respectively, in tangential variable $x$ and normal variable $y.$  
The  main result can be stated as below. 
  
\begin{theorem}\label{th2d}
	Let $1\leq \sigma\leq 2$ and $0<\rho_0 \leq 1$.  Suppose the initial data in \eqref{hp} satisfy that  $ u_{0},u_{1}  \in X_{2\rho_0, \sigma}$,  compatible with  the boundary condition in \eqref{hp}.	  Then the 2D or 3D hyperbolic Prandtl system \eqref{hp} admits a unique solution $u \in L^\infty\big([0,T];~X_{\rho,\sigma}\big)$ for some $T>0$ and some $0<\rho<2\rho_0 $.  
\end{theorem}

\begin{remark}
	The above result shows that the well-posedness theory of hyperbolic Prandtl equation is not worse than the one for the parabolic counterpart.  For the classical Prandtl equation it is well-known 2 is the critical  Gevrey index for the well-posedness, seeing for instance \cite{MR3925144,MR2601044, lmy}. 	
	 Naturally we would ask  the critical Gevrey index for the hyperbolic version  of Prandtl equations, which remains unclear so far. 
\end{remark}

\begin{remark}
	As to be seen below, 
although the  cancellation mechanism is an efficient tool when investigating the well-posedness theory for the classical   Prandtl equations, it can not apply to the hyperbolic version. Inspired by the abstract Cauchy-Kowalewski theorem and by virtue of the hyperbolic feature,    we present in this text    a straightforward proof of the main result, basing on an elementary $L^2$ energy estimate. In particular our argument does not involve the cancellation mechanism developed for the classical Prandtl equations.  We believe the argument presented here can apply to other type  hyperbolic equations with loss of derivatives. 
\end{remark}

 \section{A priori estimate and the methodology}\label{secapriori}
 
  \subsection{A priori estimate}
  
 The key part for proving the main result Theorem \ref{th2d} is to establish a priori estimate for  the system \eqref{hp}. In fact the existence and uniqueness theory will follow  two main strategies, one refers to the construction of approximate solutions for a regularized hyperbolic Prandtl equations which is quite standard since we can apply the classical hyperbolic theory, the other refers to the derivation of an uniform estimate for these approximate solutions, and this follows from the same    argument  for proving a priori estimate.    For sake of simplicity, we will only prove  a priori estimate for regular solutions.

   \begin{assumption}\label{assmain}
	Let  $X_{\rho,\sigma}$ be the Gevrey function space  equipped with the norm $\abs{\cdot}_{\rho,\sigma}$ given in  Definition \ref{defgev}, and let       $u\in L^\infty\inner{[0, T];~X_{\rho_0,\sigma}}$ solve the hyperbolic system    \eqref{hp}   with initial data $u_0,u_1\in X_{2\rho_0,\sigma}$ for some $0<\rho_0\leq 1$.     Moreover  we  suppose  that there exists a constant $C_*$     such that   
   \begin{equation}\label{condi1}
 \forall\  t\in[0,T],\quad  \sup_{ \abs\alpha \leq  2}   \big\| \partial_x^\alpha    u(t)\big\|_{L^2}+   \sup_{ \abs\alpha \leq  2}   \big\| \partial_x^\alpha   \partial_y u(t)\big\|_{L^2} \leq  C_*,
\end{equation}
where the constant $C_*\geq 1$  depends only on $\abs{ u_0 }_{2\rho_0, \sigma}$, $\abs{ u_1 }_{2\rho_0, \sigma}$,  the Sobolev embedding  constants and the numbers $\rho_0, \ell$ that are given in Definition \ref{defgev}. 
\end{assumption}

 \begin{theorem}
 	\label{thmapri}
Let $1\leq \sigma\leq 2$. Under Assumption \ref{assmain} we can find a constant $C$,  depending only on the Sobolev embedding  constants and the numbers $\rho_0,  \ell$ that are given in Definition \ref{defgev}, such that for any $t\in [0,T]$ the following estimate
\begin{eqnarray*}
	 \abs{u(t)}_{\rho,\sigma}^2 \leq  C(\abs{ u_0 }_{2\rho_0, \sigma}^2+\abs{ u_1 }_{2\rho_0, \sigma}^2)+ C   \int_0^{t}  \inner{ \abs{u(s)}_{ \rho,\sigma}^2+\abs{u(s)}_{ \rho,\sigma}^3} ds+CC_*   \int_0^{t}  \frac{\abs{u(s)}_{ \tilde\rho,\sigma}^2}{  \tilde\rho-\rho}  ds 
\end{eqnarray*}
holds true for any pair $(\rho, \tilde\rho)$ with  $0<\rho<\tilde\rho\leq\rho_0$.
 \end{theorem}  
 
 Once we have the a priori estimate in Theorem \ref{thmapri} the existence and uniqueness for the hyperbolic system \eqref{hp} will follow,  just repeating the presentation of \cite[Section 6]{lmy} or \cite[Section 8]{MR4055987} with slight modification.  
 
 \subsection{Difficulties and Methodologies} As for the classical   Prandtl equation, the main difficulty arises from the loss of tangential derivatives due the absence of diffusion in $x$ variable.  So far it is well-explored to overcome the loss of derivatives when investigating the well-posedness  of the classical Prandtl equation, seeing for instance the recent works \cite{MR3327535, MR3925144, lmy, MR3385340}, where the main idea involved  is the cancellation mechanism initiated by \cite{MR3327535,MR3385340}. Note that these cancellation technique can not apply to the hyperbolic Prandtl system.  Precisely when performing energy estimate for system \eqref{hp} we have 
 \begin{eqnarray*}
 	\frac{1}{2}\frac{d}{dt}\inner{\norm{\partial_t u}_{L^2}^2+\norm{\partial_y u}_{L^2}^2}+\norm{\partial_t u}_{L^2}^2=-(u\partial_x u, \ \partial_t u)_{L^2}-(v\partial_y u, \ \partial_t u)_{L^2},
 \end{eqnarray*}  
 where the loss of one order tangential derivatives occurs in the   terms on the right side with the main difficulty arising from the first one.  To overcome the loss of derivatives it relies on the observation that  we only lose half order rather than one order derivatives  due to the hyperbolic feature.    Our argument is inspired by the  abstract Cauchy-Kowalewski theorem,  whose statement in general Banach scales can be found in  \cite{MR966397, MR322321}  as well as the references therein; see \cite{lmy, MR4055987, MR1617542}  as well for its application  to the Well-posedness theory of classical Prandtl equations  in analytic or Gevrey  spaces.   Consider    the Cauchy  problem  \begin{eqnarray*}
 	\partial_t h=F(t, h, \partial_x h),\quad  h|_{t=0}=h_0,
 \end{eqnarray*}
 and for given analytic data  $F$ and $h_0$,   we may only expect, by virtue of  the classical Cauchy-Kowalewski theorem,  the local existence and uniqueness in analytic space for the Cauchy problem above   since it loses in general one order derivative.  Moreover as far as  the following  Cauchy problem
  \begin{eqnarray}\label{g2}
 	\partial_t^2 h=F(t, h,\nabla h),\quad h|_{t=0}=h_0, \ ~  \partial _th |_{t=0}=h_1,
 	 \end{eqnarray}
 	 is concerned, 
   the existence theory can be extended to any Gevrey space once the Gevrey index $\leq 2$ by using abstract Cauchy-Kowalewski theorem in Gevrey class.  We refer to \cite{lmy} for the application of the abstract Cauchy-Kowalewski theorem of Gevrey version by exploring the intrinsic structure similar as in \eqref{g2}.  Note  the hyperbolic version of Prandtl system \eqref{hp} may be viewed as the type of \eqref{g2} due to the hyperbolic factor $\partial_t^2$ in \eqref{hp}.

\section{Proof of the a priori estimate}\label{subap} 

This part is devoted to proving the a priori estimate stated in Theorem \ref{thmapri}. Without loss of generality we only present in detail the proof of Theorem \ref{thmapri} for $n=2$, and the three-dimensional case  can be treated in the same way. 

\begin{proof}
	[Proof of Theorem \ref{thmapri} (Two-dimensional case)]
We proceed through several steps to deal with the terms involved in Definition \ref{defgev} of $\abs{u}_{\rho,\sigma}$. In the following argument we suppose $t$ is  fixed with $0<t\leq T$. To simplify the notation we use the capital letters $ C$  to denote some generic constant that may vary from line to line, depending only on the Sobolev embedding constants and the numbers $\rho_0,\ell$ given in Definition \ref{defgev}  but independent of the constant $C_*$ in \eqref{condi1} and the order of derivatives denoted by $m$. 

\medskip
\noindent{\it Step (a)}.  In this step, we conclude that for any $m\geq 7$ and any  pair $(\rho,\tilde\rho)$ with $0<\rho<\tilde\rho\leq\rho_0$,  
\begin{equation}\label{sta++}
	\frac{\rho^{2(m-7)}} {[\inner{m-7}!]^{2\sigma}}m^2\norm{\comi y^\ell\partial_x^mu(t)}_{L^2}^2 \leq C  \abs{u_0}_{2\rho_0,\sigma}^2+ C \int_0^{t}  \frac{\abs{u(s)}_{ \tilde\rho,\sigma}^2}{  \tilde\rho-\rho}  ds.
\end{equation}  
In fact, 
 it follows  from Definition \ref{defgev} of $\abs{u}_{r,\sigma}$  that,  for any  integer $j\geq 0$  and  for any $0<r\leq \rho_0$, 
  \begin{equation}\label{uint+}
 	   \norm{\comi y^{\ell}\partial_t\partial_x^j  u }_{L^2} + \norm{\comi y^{\ell}\partial_y \partial_x^j  u }_{L^2}+j\norm{ \comi y^{\ell}\partial_x^j  u }_{L^2}  \leq   
 	 \left\{
 	 \begin{aligned} 
 	&  \frac{[(j-7)!]^\sigma}{r^{j-7}}\abs{u}_{r,\sigma},\quad {\rm if}~j \geq 7,\\
 	 & 	\abs{u}_{r,\sigma}, \quad {\rm if}~j \leq 6.
 	 \end{aligned}
\right.
 	  \end{equation}
As a result, 
\begin{equation}\label{deq}
	\begin{aligned}
		&m^2\int_0^t    \norm{\comi y^\ell\partial_x^mu(s)}_{L^2} \norm{ \comi y^\ell\partial_t\partial_x^mu(s)}_{L^2} ds\\
		&\leq  \int_0^t  m   \frac{[(m-7)!]^{2\sigma}}{\tilde\rho^{2(m-7)}} \abs{u(s)}_{ \tilde\rho,\sigma}^2ds\leq C  \frac{[\inner{m-7}!]^{2\sigma}}{\rho^{2(m-7)}}  \int_0^{t}  \frac{\abs{u(s)}_{ \tilde\rho,\sigma}^2}{  \tilde\rho-\rho}  ds,
	\end{aligned}
\end{equation}
the last inequality holding because for any pair $(\rho,\tilde\rho)$ with $0<\rho<\tilde\rho\leq \rho_0\leq 1,$
\begin{equation}\label{kefact}
   \frac{m}{\tilde\rho^{2(m-7)}} \leq 	\frac{m}{\tilde \rho} \frac{1}{\tilde\rho^{2(m-7)}}  =  \frac{1}{\rho^{2(m-7)}}\frac{m}{\tilde\rho}\frac{\rho^{2(m-7)}}{\tilde\rho^{2(m-7)}}\leq  \frac{C}{\rho^{2(m-7)}} \frac{m-7 }{\tilde\rho}\Big(\frac{\rho}{\tilde\rho}\Big)^{m-7}    \leq   \frac{C}{\rho^{2(m-7)}}  \frac{1}{\tilde\rho-\rho} 
\end{equation}
due to the fact that the inequality $ 
  k r^k\leq \frac{1}{1-r}$ holds true for any integer $k\geq 0$ and any number $0<r<1.$  On the other hand, observing $\rho\leq\rho_0$, 
  \begin{multline*}
  		m^2\norm{\comi y^\ell\partial_x^mu_0}_{L^2}^2\leq m^2\frac{[(m-7)!]^{2\sigma}}{ (2\rho_0)^{2(m-7)}}\abs{u_0}_{2\rho_0,\sigma}^2\\
  		\leq m^2\frac{\rho^{2(m-7)}}{ (2\rho_0)^{2(m-7)}}\frac{[(m-7)!]^{2\sigma}}{  \rho^{2(m-7)}}\abs{u_0}_{2\rho_0,\sigma}^2 \leq C \frac{[(m-7)!]^{2\sigma}}{  \rho^{2(m-7)}}\abs{u_0}_{2\rho_0,\sigma}^2.
  	\end{multline*}
  Combining the above inequality  and \eqref{deq} with the fact that
\begin{eqnarray*}
	\begin{aligned}
		m^2\norm{\comi y^\ell\partial_x^mu(t)}_{L^2}^2&= m^2\norm{\comi y^\ell\partial_x^mu_0}_{L^2}^2+m^2\int_0^t\frac{d}{ds} \norm{\comi y^\ell\partial_x^mu(s)}_{L^2}^2 ds\\
		&= m^2\norm{\comi y^\ell\partial_x^mu_0}_{L^2}^2+2m^2\int_0^t  {\rm Re} \inner{\comi y^\ell\partial_x^mu(s), \ \comi y^\ell(\partial_t\partial_x^mu)(s)}_{L^2}  ds\\
		&\leq m^2\norm{\comi y^\ell\partial_x^mu_0}_{L^2}^2+2m^2\int_0^t   \norm{\comi y^\ell\partial_x^mu(s)}_{L^2}\norm{  \comi y^\ell(\partial_t\partial_x^mu)(s)}_{L^2}  ds\\
		&\leq C  \frac{[\inner{m-7}!]^{2\sigma}}{\rho^{2(m-7)}}\Big( \abs{u_0}_{2\rho_0,\sigma}^2 +\int_0^{t}  \frac{\abs{u(s)}_{ \tilde\rho,\sigma}^2}{  \tilde\rho-\rho}  ds\Big), 
	\end{aligned}
\end{eqnarray*} 	  
 	  we obtain the assertion \eqref{sta++}.

\medskip
\noindent{\it Step (b)}. This step is devoted to proving that, for any $m\geq 7$ and any  pair $(\rho,\tilde\rho)$ with $0<\rho<\tilde\rho\leq\rho_0$,  
\begin{equation}\label{sta}
\begin{aligned}
	&\frac{\rho^{2(m-7)}} {[\inner{m-7}!]^{2\sigma}}\inner{\norm{\comi y^\ell\partial_t\partial_x^mu(t)}_{L^2}^2+\norm{\comi y^\ell\partial_y\partial_x^mu(t)}_{L^2}^2} \\
	& \leq  C\inner{\abs{u_0}_{2\rho_0,\sigma}+\abs{u_1}_{2\rho_0,\sigma}}  +   C  \int_0^{t} \inner{ \abs{u(s)}_{ \rho,\sigma}^2+\abs{u(s)}_{ \rho,\sigma}^3}ds+ CC_* \int_0^{t}  \frac{\abs{u(s)}_{ \tilde\rho,\sigma}^2}{  \tilde\rho-\rho}  ds,
	\end{aligned}
\end{equation}  
where $C_*$ is the number given in \eqref{condi1}. 	  
To do so, applying  $\partial_x^m$ to the first equation in \eqref{hp} gives
\begin{equation}\label{fuc} 
\begin{aligned}
	 \big(\partial_t^2+\partial_t -\partial_y^2\big)   \partial_x^{m} u  =   -\sum_{j=0}^{m}{m\choose j}  \Big[(\partial_x^j u) \partial_{x}^{m-j+1}u   +(\partial_x^j v)\partial_{x}^{m-j}\partial_yu \Big].
	\end{aligned}
\end{equation} 
Furthermore, multiplying the both sides of  \eqref{fuc} by  $\comi y^{2\ell} \partial_t\partial_x^mu $ and then integrating over $[0, t]\times \mathbb R_+^2$, we obtain, observing the initial-boundary conditions in \eqref{hp},
\begin{equation}\label{uma++} 
\begin{aligned}
&\frac{1}{2}\norm{\comi y^{ \ell}\partial_t\partial_x^{m} u (t)}_{L^2}^2 +\frac{1}{2}\norm{\comi y^{ \ell}\partial_y\partial_x^{m} u (t)}_{L^2}^2
+\int_0^{t} \norm{\comi y^{ \ell}\partial_t\partial_x^{m} u (s)}_{L^2}^2ds 
\\&
= \frac{1}{2}\norm{\comi y^{ \ell} \partial_x^{m} u_1}_{L^2}^2 +\frac{1}{2}\norm{\comi y^{ \ell}\partial_y\partial_x^{m} u_0}_{L^2}^2 \\
&\quad -\sum_{j=0}^{m}{m\choose j}\int_0^{t} \Big( \comi y^{\ell} \big[(\partial_x^j u) \partial_{x}^{m-j+1}u   +(\partial_x^j v)\partial_{x}^{m-j}\partial_yu \big]  ,\ \comi y^{\ell}  \partial_t\partial_x^{m}u \Big)_{L^2}ds\\
& \quad - \int_0^t\inner{  \partial_y\partial_x^{m}u, \    (\partial_y\comi y^{2\ell})\partial_t\partial_x^{m}u}_{L^2}ds.
\end{aligned}
\end{equation}	
It follows from Definition \ref{defgev} that
\begin{multline*}
	\frac{1}{2}\norm{\comi y^{ \ell} \partial_x^{m} u_1}_{L^2}^2 +\frac{1}{2}\norm{\comi y^{ \ell}\partial_y\partial_x^{m} u_0}_{L^2}^2\\
	\leq  \frac{[(m-7)!]^{2\sigma}}{\rho^{2(m-7)}}\big(\abs{ u_0}_{\rho,\sigma}^2 + \abs{ u_1}_{\rho,\sigma}^2\big)\leq  \frac{[(m-7)!]^{2\sigma}}{\rho^{2(m-7)}}\big(\abs{ u_0}_{2\rho_0,\sigma}^2 + \abs{ u_1}_{2\rho_0,\sigma}^2\big).
	\end{multline*}
	 As for the last term on the right of \eqref{uma++} we use \eqref{uint+} to conclude 
 	  \begin{equation*}\label{fet}
 	  	\Big|\int_0^t\inner{  \partial_y\partial_x^{m}u, \    (\partial_y\comi y^{2\ell})\partial_t\partial_x^{m}u}_{L^2}ds\Big|\leq C\frac{[(m-7)!]^{2\sigma}}{\rho^{2(m-7)}}\int_0^t \abs{u(s)}_{\rho,\sigma}^2ds.
 	  \end{equation*}
 	  Then the desired  estimate \eqref{sta} will follow if we can prove that
 	  \begin{multline}\label{ges}
	\sum_{j=0}^{m}{m\choose j}\Big|\int_0^{t} \Big( \comi y^{\ell} \big[(\partial_x^j u) \partial_{x}^{m-j+1}u   +(\partial_x^j v)\partial_{x}^{m-j}\partial_yu \big]  ,\ \comi y^{\ell}  \partial_t\partial_x^{m}u \Big)_{L^2}ds\Big|\\
	\leq  C  \int_0^{t} \abs{u(s)}_{ \rho,\sigma}^3 ds+ CC_* \int_0^{t}  \frac{\abs{u(s)}_{ \tilde\rho,\sigma}^2}{  \tilde\rho-\rho}  ds.
\end{multline} 
 	To do so we    write
 \begin{equation}\label{i2+}
 \begin{aligned}
& \sum_{j=0}^{m}{m\choose j}\Big|\int_0^{t} \Big( \comi y^{\ell} \big[(\partial_x^j u) \partial_{x}^{m-j+1}u   +(\partial_x^j v)\partial_{x}^{m-j}\partial_yu \big]  ,\ \comi y^{\ell}  \partial_t\partial_x^{m}u \Big)_{L^2}ds\Big|\\
 & \leq  \sum_{j=0}^{m}{{m}\choose j} \int_0^t \norm{  \comi y^{\ell}(\partial_x^j u )\partial_x^{m-j+1}u }_{L^2}\norm{  \comi y^{\ell}\partial_t \partial_x^{m}u }_{L^2} ds\\
 	&\quad +\sum_{j=0}^{m}{{m}\choose j} \int_0^t  \norm{ \comi y^{\ell} (\partial_x^j v )\partial_y \partial_x^{m-j}u }_{L^2}\norm{ \comi y^{\ell} \partial_t \partial_x^{m}u }_{L^2} ds.
 \end{aligned}
 \end{equation}
For the first term on the right side we have,  denoting  by $[p] $   the largest integer less than or equal to $p,$  
 \begin{equation}\label{estms}
 	\begin{aligned}
 		\sum_{j=0}^{m}{{m}\choose j}  \norm{\comi y^{\ell} (\partial_x^j u) \partial_x^{m-j+1}u }_{L^2} 
	\leq 	&\sum_{j=0}^{[m/2]}{{m}\choose j}  \norm{ \partial_x^j u }_{L^\infty} \norm{ \comi y^{\ell} \partial_x^{m-j+1}u }_{L^2} \\
	&+ \sum_{j= [m/2]+1}^m {{m}\choose j}  \norm{ \comi y^{\ell}\partial_x^j u}_{L^2}\norm{ \partial_x^{m-j+1} u }_{L^\infty}.
 	\end{aligned}
 \end{equation}
Using again \eqref{uint+} and the  Sobolev embedding  inequality 
\begin{equation*}
	\label{soblev}
	 \norm{F}_{L^\infty}\leq C\Big( \norm{  F}_{H_x^2({L_y^2)}}+ \norm{  \partial_y F}_{H_x^2({L_y^2)}}\Big)
	\end{equation*}
	which holds true for $x\in \mathbb R$ or $x\in\mathbb R^2$,
we compute, in view of \eqref{condi1},
	\begin{equation}\label{spe}
	\begin{aligned}
 	&\sum_{j=0}^{[m/2]}{{m}\choose j}  \norm{ \partial_x^j u }_{L^\infty} \norm{  \comi y^{\ell}\partial_x^{m-j+1}u }_{L^2}  \leq   C \sum_{j=5}^{[m/2]}\frac{m!} {j!(m-j)!} \frac{[(j-5)!]^\sigma}{ \rho^{j-5}} \frac{[(m-j-6)!]^\sigma}{(m-j+1)\rho^{m-j-6}} \abs{u}_{\rho,\sigma}^2 \\
   &\qquad\qquad+ C  \sum_{1\leq j\leq 4} \frac{m!} {j!(m-j)!}  \frac{[(m-j-6)!]^\sigma}{(m-j+1) \rho^{m-j-6}} \abs{u}_{   \rho,\sigma}^2+    C_*    \frac{[(m -6)!]^\sigma}{(m+1)\tilde\rho^{m-6}} \abs{u}_{ \tilde \rho,\sigma}, 
   \end{aligned}
\end{equation}
with $C_*$ the constant given in \eqref{condi1}.
Direct verification shows, observing $1\leq \sigma\leq 2,$ 
 \begin{equation*}
 \sum_{1\leq j\leq 4} \frac{m!} {j!(m-j)!}  \frac{[(m-j-6)!]^\sigma}{(m-j+1) \rho^{m-j-6}} \abs{u}_{   \rho,\sigma}^2   \leq C     \frac{[(m-7)!]^\sigma}{ \rho^{m-7}} \abs{u}_{  \rho,\sigma} ^2,  
 \end{equation*}
 \begin{eqnarray*}
 	\frac{[(m -6)!]^\sigma}{(m+1)\tilde\rho^{m-6}} \abs{u}_{ \tilde \rho,\sigma}\leq C\frac{m^{\sigma-1}}{\tilde\rho} \frac{[(m -7)!]^\sigma}{ \tilde\rho^{m-7}} \abs{u}_{ \tilde \rho,\sigma}\leq C\frac{m}{\tilde\rho} \frac{[(m -7)!]^\sigma}{\tilde\rho^{m-7}} \abs{u}_{ \tilde \rho,\sigma}
 \end{eqnarray*}
and   
\begin{eqnarray*}
\begin{aligned}
&  \sum_{j=5}^{[m/2]}\frac{m!} {j!(m-j)!} \frac{[(j-5)!]^\sigma}{ \rho^{j-5}} \frac{[(m-j-6)!]^\sigma}{(m-j+1)\rho^{m-j-6}} \abs{u}_{\rho,\sigma}^2\\
	&\leq  C \frac{   \abs{u}_{\rho,\sigma}^2}{\rho^{m-7}}  \sum_{j=5}^{[m/2]} \frac{m! [(j-5)!]^{\sigma-1} [(m-j-6)!]^{\sigma-1}} {j^5(m-j)^7} \\   
	& \leq   C\frac{   \abs{u}_{\rho,\sigma}^2}{\rho^{m-7}} \sum_{j=5}^{[m/2]} \frac{(m-7)! m^7} {j^5(m-j)^7 }   [(m-7)!]^{\sigma-1}  
	 \leq  C \frac{  [(m-7)!]^{\sigma} } {\rho^{m-7}}\abs{u}_{\rho,\sigma}^2. 
	 \end{aligned}
\end{eqnarray*}
Substituting  the above inequalities into \eqref{spe} gives
\begin{equation}\label{jm}
	\sum_{j=1}^{[m/2]}{{m}\choose j}  \norm{ \partial_x^j u }_{L^\infty} \norm{ \comi y^{\ell} \partial_x^{m-j+1}u }_{L^2} \leq C \frac{  [(m-7)!]^{\sigma} } {\rho^{m-7}}\abs{u}_{\rho,\sigma}^2+CC_*\frac{m}{\tilde\rho}  \frac{[(m-7)!]^\sigma}{\tilde\rho^{m-7}} \abs{u}_{ \tilde\rho,\sigma}.
\end{equation}
Following a similar argument for proving \eqref{jm},   we also have 
\begin{eqnarray*}
	\sum_{j=[m/2]+1}^m{{m}\choose j}  \norm{\comi y^{\ell} \partial_x^j u }_{L^2} \norm{  \partial_x^{m-j+1}u }_{L^\infty} \leq C \frac{  [(m-7)!]^{\sigma} } {\rho^{m-7}}\abs{u}_{\rho,\sigma}^2.
\end{eqnarray*}
Substituting the above inequality and \eqref{jm} into \eqref{estms},
we conclude
\begin{eqnarray*}
	\sum_{j=0}^{m}{{m}\choose j}  \norm{ \comi y^{\ell} (\partial_x^j u) \partial_x^{m-j+1}u }_{L^2} \leq C \frac{  [(m-7)!]^{\sigma} } {\rho^{m-7}}\abs{u}_{\rho,\sigma}^2+CC_*\frac{m}{\tilde\rho}  \frac{[(m-7)!]^\sigma}{\tilde\rho^{m-7}} \abs{u}_{ \tilde\rho,\sigma}.
\end{eqnarray*}
This, together with \eqref{uint+}, yields
\begin{eqnarray}\label{eng1}
\begin{aligned}
	& \int_0^t \sum_{j=0}^{m}{{m}\choose j}\norm{ \comi y^{\ell} (\partial_x^j u ) \partial_x^{m-j+1}u }_{L^2}\norm{  \comi y^{\ell}\partial_t \partial_x^{m}u }_{L^2} ds\\
	&\leq C \frac{[\inner{m-7}!]^{2\sigma}}{\rho^{2(m-7)}}  \int_0^{t}   \abs{u(s)}_{ \rho,\sigma}^3 ds+CC_*  [\inner{m-7}!]^{2\sigma}   \int_0^{t} \frac{m}{\tilde\rho}   \frac{\abs{u(s)}_{ \tilde\rho,\sigma}^2}{  \tilde \rho^{2(m-7)}}  ds\\
	&\leq C \frac{[\inner{m-7}!]^{2\sigma}}{\rho^{2(m-7)}}  \int_0^{t}   \abs{u(s)}_{ \rho,\sigma}^3 ds+CC_* \frac{[\inner{m-7}!]^{2\sigma}}{\rho^{2(m-7)}}  \int_0^{t}  \frac{\abs{u(s)}_{ \tilde\rho,\sigma}^2}{  \tilde\rho-\rho}  ds, 
	\end{aligned}
\end{eqnarray}
due to \eqref{kefact}.
Similarly we follow a similar argument as  above  to conclude
\begin{eqnarray}\label{eng2}
\begin{aligned}
	& \int_0^t  \sum_{j=0}^{m}{{m}\choose j} \norm{\comi y^{\ell} (\partial_x^j v) \partial_y \partial_x^{m-j}u }_{L^2}\norm{\comi y^{\ell} \partial_t \partial_x^{m}u }_{L^2} ds\\  
	&\leq  \int_0^t  \sum_{j=0}^{[m/2]}{{m}\choose j} \norm{ \partial_x^j v }_{L^\infty}\norm{\comi y^{\ell}  \partial_y \partial_x^{m-j}u }_{L^2}\norm{\comi y^{\ell} \partial_t \partial_x^{m}u }_{L^2} ds\\
	&\quad+ \int_0^t \sum_{j= [m/2]+1}^m {{m}\choose j}  \norm{ \partial_x^j v }_{L_x^2(L_y^\infty)}\norm{\comi y^{\ell}  \partial_y \partial_x^{m-j}u }_{L_x^\infty(L_y^2)}\norm{\comi y^{\ell} \partial_t \partial_x^{m}u }_{L^2} ds\\
	&\leq   C \frac{[\inner{m-7}!]^{2\sigma}}{\rho^{2(m-7)}}  \int_0^{t}   \abs{u(s)}_{ \rho,\sigma}^3 ds +CC_* \frac{[\inner{m-7}!]^{2\sigma}}{\rho^{2(m-7)}}  \int_0^{t}  \frac{\abs{u}_{ \tilde\rho,\sigma}^2}{  \tilde\rho-\rho}  ds.
	\end{aligned}
	\end{eqnarray}
where we have used the inequality that if $ \ell>1/2 $ and $ \partial_{x}u+\partial_{y}v=0 $ as well as $ v|_{ y=0}=0, $   
\begin{eqnarray*}
\begin{aligned}
\norm{\partial_{x}^{j} v}_{L_{y}^{\infty} L_{x}^{2}}
=&\norm{\int_{0}^{y}\partial_{\tilde y}\partial_{x}^{j} v(t,x,\tilde{y})d\tilde{y}}_{L_{y}^{\infty} L_{x}^{2}}
\leq\norm{\int_{0}^{y}\partial_{x}^{j+1} u(t,x,\tilde{y})d\tilde{y}}_{L_{y}^{\infty} L_{x}^{2}}\\
\leq&\int_{0}^{\infty}\norm{\comi{\tilde y}^{\ell}  \partial_{x}^{j+1} u(\cdot, \tilde{y})}_{L_{x}^{2}} \comi{\tilde y}^{-\ell} d \tilde{y}\\
\leq& \norm{\comi{y}^{\ell}  \partial_{x}^{j+1} u}_{L_{x, y}^{2}}\norm{\comi{y}^{-\ell}}_{L_{y}^{2}}
\leq C\norm{\comi{y}^{\ell}  \partial_{x}^{j+1} u}_{L_{x, y}^{2}},
\end{aligned}
\end{eqnarray*}
for all integers $j \geq 0$, all $t \in[0, T]$, and some positive constant $C=C(\ell)$.
Combining the two estimates \eqref{eng1} , \eqref{eng2} with \eqref{i2+}, we obtain \eqref{ges}. Thus we can get
the desired estimate \eqref{sta}.

\medskip
\noindent{\it Step (c)}   We combine the estimates \eqref{sta++} and \eqref{sta} to conclude, for any $m\geq 7$ and any  pair $(\rho,\tilde\rho)$ with $0<\rho<\tilde\rho\leq\rho_0$,
\begin{multline*}	\frac{\rho^{2(m-7)}} {[\inner{m-7}!]^{2\sigma}}\inner{\norm{\comi y^\ell\partial_t\partial_x^mu(t)}_{L^2}^2+\norm{\comi y^\ell\partial_y\partial_x^mu(t)}_{L^2}^2+m^2\norm{\comi y^\ell\partial_x^mu(t)}_{L^2}^2} \\ \leq CC\inner{\abs{u_0}_{2\rho_0,\sigma}+\abs{u_1}_{2\rho_0,\sigma}}  +   C  \int_0^{t} \inner{ \abs{u(s)}_{ \rho,\sigma}^2+\abs{u(s)}_{ \rho,\sigma}^3}ds+ CC_* \int_0^{t}  \frac{\abs{u(s)}_{ \tilde\rho,\sigma}^2}{  \tilde\rho-\rho}  ds.
\end{multline*}
It can be checked straightforwardly that the estimate
\begin{multline*}	
\sum_{0\leq m\leq 6}\inner{\norm{\comi y^\ell\partial_t\partial_x^mu(t)}_{L^2}^2+\norm{\comi y^\ell\partial_y\partial_x^mu(t)}_{L^2}^2+m^2\norm{\comi y^\ell\partial_x^mu(t)}_{L^2}^2} \\ \leq  C\inner{\abs{u_0}_{2\rho_0,\sigma}+\abs{u_1}_{2\rho_0,\sigma}}  +   C  \int_0^{t} \inner{ \abs{u(s)}_{ \rho,\sigma}^2+\abs{u(s)}_{ \rho,\sigma}^3}ds+ CC_* \int_0^{t}  \frac{\abs{u(s)}_{ \tilde\rho,\sigma}^2}{  \tilde\rho-\rho}  ds
\end{multline*}
holds true for  any  pair $(\rho,\tilde\rho)$ with $0<\rho<\tilde\rho\leq\rho_0$. As a result, we combine the two estimates above to complete the proof of Theorem \ref{thmapri} for the two-dimensional case. 
\end{proof}

\begin{proof}
	[Proof of Theorem \ref{thmapri} (Three-dimensional case)] It can be treated in the same way as that for two-dimensional case. In fact following the argument above gives that for any $m\geq 7$ we have,  recalling  $x=(x_1,x_2)$,
\begin{multline*}	\frac{\rho^{2(m-7)}} {[\inner{m-7}!]^{2\sigma}}\inner{\norm{\comi y^\ell\partial_t\partial_{x_j}^mu(t)}_{L^2}^2+\norm{\comi y^\ell\partial_y\partial_{x_j}^mu(t)}_{L^2}^2+m^2\norm{\comi y^\ell\partial_{x_j}^mu(t)}_{L^2}^2} \\ \leq  C\inner{\abs{u_0}_{2\rho_0,\sigma}+\abs{u_1}_{2\rho_0,\sigma}}  +   C  \int_0^{t} \inner{ \abs{u(s)}_{ \rho,\sigma}^2+\abs{u(s)}_{ \rho,\sigma}^3}ds+ CC_* \int_0^{t}  \frac{\abs{u(s)}_{ \tilde\rho,\sigma}^2}{  \tilde\rho-\rho}  ds,
\end{multline*}
with $j=1$ or $2$.  This along with the fact that
\begin{eqnarray*}
	\norm{\partial_x^\alpha F}_{L^2}\leq C \inner{\norm{\partial_{x_1}^{\abs\alpha} F}_{L^2}+\norm{\partial_{x_2}^{\abs\alpha} F}_{L^2}} 
\end{eqnarray*}
yields
\begin{multline*}
	\sup_{ \abs\alpha\geq 7} \frac{\rho^{\abs\alpha-7}}{[\inner{\abs\alpha-7}!]^{\sigma}} \Big(  \big\|\comi y^{\ell}  \partial_t\partial_x^\alpha f(t)\big\|_{L^2}+ \big\|\comi y^{\ell }  \partial_y\partial_x^\alpha      f(t)\big\|_{L^2}+\abs\alpha\norm{\comi y^{\ell } \partial_x^\alpha      f(t)}_{L^2}\Big)\\
	\leq  C\inner{\abs{u_0}_{2\rho_0,\sigma}+\abs{u_1}_{2\rho_0,\sigma}}  +   C  \int_0^{t} \inner{ \abs{u(s)}_{ \rho,\sigma}^2+\abs{u(s)}_{ \rho,\sigma}^3}ds+ CC_* \int_0^{t}  \frac{\abs{u(s)}_{ \tilde\rho,\sigma}^2}{  \tilde\rho-\rho}  ds.
\end{multline*}
The case when $\abs\alpha\leq 6$ is  straightforward. Then we have completed the proof of Theorem \ref{thmapri} for the three-dimensional case.
\end{proof}
     
\appendix
\section{Derivation of the boundary layer system}\label{secderivation}

In this section, we only give  the derivation of  two-dimensional (2D) boundary layer system in  \eqref{mhdsys+++} since the 3D can be derived in the same way.  We consider the following 2D hyperbolic version of Navier-Stokes  system in  $\Omega: =\mathbb R\times\mathbb R_+=\big\{ (x,y)\in\mathbb R_{+}^2;\  x\in\mathbb R,  y>0\big \}$
\begin{equation}\label{mhdsys}
\left\{
\begin{aligned}
&\eta\partial_t^2 \bm  u^{\eps}+\partial_t \bm  u^{\eps} +(\bm  u^{\eps} \cdot\nabla) \bm u^{\eps}-\eps^2 \Delta \bm u^{\eps}+\nabla   p^{\eps} =0,\\
&\nabla\cdot \bm u^{\eps}=0,\\
&{\bm u^{\eps}}|_{t=0}= {\bm u_0} ,\quad {\partial_t \bm u^{\eps}}|_{t=0}={\bm u_{1}},
\end{aligned}
\right.
\end{equation}
where $ \eta=1 $ and $ \bm u^{\eps}=(u^{\eps},v^{\eps}) $ denotes velocity fields, $ p^{\eps} = p^{\eps}(t, x, y) $ denotes the scalar pressure.  The above system is complemented with no-slip boundary conditions on the
velocity  fields,   that is, 
\begin{equation}\label{cond2}
( u^{\eps},v^{\eps})|_{ y=0}= (0,0). 
\end{equation}

\noindent
{\bf Away from the boundary:} We construct the approximate solution by the following expansion
  \begin{eqnarray*}\label{aexps}
  \left\{
  \begin{aligned}
   u^{\eps}\left(t, x,y\right)&\sim\sum_{j=0}^{2} \eps^{j}u^{I, j}(t,x,y) + o ( \eps  );\\
 v^{\eps}\left(t, x,y\right)&\sim\sum_{j=0}^{2} \eps^{j}v^{I, j}(t,x,y) + o (\eps  );\\
	  p^{\eps}\left(t,x, y\right)&\sim\sum_{j=0}^{2}\eps^{j} p^{I, j}(t,x, y)+ o( \eps  ).
  \end{aligned}
  \right.
  \end{eqnarray*}
Plugging the above expansion into the hyperbolic Navier-Stokes  system \eqref{mhdsys}, and letting $ \eps\rightarrow 0 $, then matching the $ O(\eps^0) $ term, we
find that $ (u^{I, 0}(t,x,y), v^{I, 0}(t,x,y) ,p^{I, 0}(t,x,y)) $ should satisfy the hyperbolic Euler equations  
 \begin{equation}\label{imhdsys}
\left\{
\begin{aligned}
&\partial_t^2 \bm  u^{I, 0}+\partial_t \bm  u^{I, 0} +(\bm  u^{I, 0} \cdot\nabla) \bm u^{I, 0} +\nabla   p^{I, 0} =0,\\
&\nabla\cdot \bm u^{I, 0}=0,\\
\end{aligned}
\right.
\end{equation}  
Naturally, we endow \eqref{imhdsys} with the same initial data as viscous flow in \eqref{mhdsys}. That is 
\begin{equation*}
{\bm u^{I, 0}}|_{t=0}=(u_0,v_0)(x,y),\quad {\partial_t \bm u^{I, 0}}|_{t=0}=(u_1,v_1)(x,y),
\end{equation*}
 For well-posedenss of the system \eqref{imhdsys}, we may impose a homogeneous Dirichlet boundary condition for the normal components of velocity field:
\begin{equation}\label{cond3}
v^{I, 0} |_{y=0}=0.
\end{equation}
 
 \noindent
 {\bf Near the boundary $ y = 0 $:} Comparing the boundary conditions \eqref{cond2} with \eqref{cond3},  there is a mismatch between
 the tangential component $ u^{\eps}(t,x,y) $ and $ u^{I, 0}(t,x,y) $ on the boundary $ \{y=0\} $. 
 To derive the equations for boundary layers, 
 we carry out the multi-scale analysis as Prandtl in \cite{MR1697762}. 
 We suppose that the solution  $ (u^{\eps},v^{\eps}) $ of the hyperbolic Navier-Stokes  system \eqref{mhdsys} accepts the following multi-scale formal expansions
 \begin{equation}\label{nexps}
 \left\{
 \begin{aligned}
 u^{\eps}\left(t, x,y\right)&\sim\sum_{j=0}^{2} \eps^{j}\left[u^{I, j}(t,x,y)+u^{B, j}(t,x, \frac{y}{\eps})\right] + o ( \eps  )
 ;\\
 v^{\eps}\left(t, x,y\right)&\sim\sum_{j=0}^{2} \eps^{j}\left[v^{I, j}(t,x,y)+v^{B, j}(t,x, \frac{y}{\eps})\right]
  + o (\eps  );\\
 p^{\eps}\left(t,x, y\right)&\sim\sum_{j=0}^{2}\eps^{j} \left[p^{I, j}(t,x, y)+ p^{B, j}(t,x,  \frac{y}{\eps})\right]+ o ( \eps  ).
 \end{aligned}
 \right.
 \end{equation} 
 where $(u^{I, j},v^{I, j})$ are inner functions, 
 $(u^{B, j},v^{B, j})$  and $p^{B, j}$ are the boundary layer functions near the boundary $ y=0. $ Moreover, the boundary layer functions $u^{B, j},v^{B, j},p^{B, j}$ exponentially go to zero as $\tilde{y}=\frac{y}{\eps}\rightarrow+\infty $ $ (\eps\rightarrow 0).$
 
 First, we substitute the ansatz given in \eqref{nexps} into the first equation of \eqref{mhdsys} to obtain
 
 \begin{equation}\label{cks}
 \begin{aligned}
 &\sum_{j\geq 0} \eps^{j}\partial_t^2(u^{I, j}+u^{B, j})
 +\sum_{j\geq 0} \eps^{j}\partial_t(u^{I, j}+u^{B, j})+
 \sum_{j\geq 0} \eps^{j}\sum_{k=0}^{j}(u^{I, k}+u^{B, k})(\partial_x  u^{I, j-k}+\partial_x  u^{B, j-k}) 
 \\&\quad
 +
 \sum_{j\geq 0} \eps^{j}\sum_{k=0}^{j}\big[(v^{I, k}+v^{B, k})\partial_yu^{I, j-k} \big]+\sum_{j\geq 0} \eps^{j}\sum_{k=0}^{j+1}(v^{I, k}+v^{B, k})\partial_{\tilde{y}}u^{B, j-k+1}+ 
 \frac{1}{\eps}v^{I, 0}\partial_{\tilde{y}}u^{B, 0}
 \\&\quad
 +\sum_{j\geq 0} \eps^{j}\partial_x p^{I, j}+\sum_{j\geq 0} \eps^{j}\partial_x p^{B, j}
 =
 \sum_{j\geq 0} \eps^{j}\big(\eps^{2}\partial_x^2u^{I, j} +\eps^{2}\partial_y^2u^{I, j}+\eps^{2}\partial_x^2u^{B, j}+\partial_{\tilde{y}}^2u^{B, j}\big),\\
 &\sum_{j\geq 0} \eps^{j}\partial_t^2(v^{I, j}+v^{B, j})
 +\sum_{j\geq 0} \eps^{j}\partial_t(v^{I, j}+v^{B, j})+
 \sum_{j\geq 0} \eps^{j}\sum_{k=0}^{j}(u^{I, k}+u^{B, k})(\partial_xv^{I, j-k}+\partial_xv^{B, j-k}) 
 \\&\quad
 +
 \sum_{j\geq 0} \eps^{j}\sum_{k=0}^{j}(v^{I, k}+v^{B, k})\partial_yv^{I, j-k} +
 \sum_{j\geq 0} \eps^{j}\sum_{k=0}^{j+1}(v^{I, k}+v^{B, k})\partial_{\tilde{y}}v^{B, j-k+1}+ 
 \frac{1}{\eps}v^{I, 0}\partial_{\tilde{y}}v^{B, 0}
 \\&\quad+
 \sum_{j\geq 0} \eps^{j}\partial_y p^{I, j}+\frac{1}{\eps}\partial_{\tilde{y}} p^{B, 0}+\sum_{j\geq 0} \eps^{j}\partial_{\tilde{y}} p^{B, j+1}\\
 &=
 \sum_{j\geq 0} \eps^{j}\big(\eps^{2}\partial_x^2v^{I, j} +\eps^{2}\partial_y^2v^{I, j}+\eps^{2}\partial_x^2v^{B, j}+\partial_{\tilde{y}}^2v^{B, j}\big),\\ 
 \end{aligned}
 \end{equation}
 Second, we substitute the ansatz given in \eqref{nexps} into the div-free condition  $ \nabla\cdot \bm u^{\eps}=0 $ of \eqref{mhdsys} to get 
 \begin{equation*}\label{divc}
 \left\{
 \begin{aligned}
 &\partial_x   u^{I, j}+\partial_y v^{I, j}=0, j\geq 0\\
 &\partial_x   u^{B, j}+\partial_{\tilde{y}} v^{B, j+1}=0, j\geq 0\\
 &\partial_{\tilde{y}} v^{B, 0}=0,\\
 \end{aligned}
 \right.
 \end{equation*}  
 thus, we have 
 \begin{equation}\label{pdivc}
  v^{B, 0}=0.\\
 \end{equation}

Third, we substitute the ansatz given in \eqref{nexps} into the  boundary condition of \eqref{mhdsys} to get
 \begin{equation*}\label{bou}
 \left\{
 \begin{aligned}
 &u^{I, j}(t,x,0)+ u^{B, j}(t,x,0)=0, j\geq 0\\
 &v^{I, j}(t,x,0)+ v^{B, j}(t,x,0)=0, j\geq 0\\
 \end{aligned}
 \right.
 \end{equation*}  
Since $ v^{I, 0}(t,x,0)=0 $, we have  
 \begin{equation*}\label{bou+}
 v^{B, 0}(t,x,0)=0.\\
 \end{equation*} 
 
Last, by using the Taylor expansion in $ y $  , we write $u^{I, 0}(t,x,y)$ as  
\begin{equation*}
	u^{I, 0}(t,x,y)=u^{I, 0}(t,x,0)+y\partial_y u^{I, 0}(t,x,0)+\frac{y^2}{2}\partial_y^2 u^{I, 0}(t,x,0)+\cdots=\overline{u^{I, 0}}+\eps\tilde {y} \overline{\partial_y u^{I, 0}}+o(\eps),
\end{equation*}
where here and below we use the notation $ \overline{f}$
to stand for the trace of a function $ f $ on the
boundary $ {y = 0} $.  Similarly, we write
\begin{equation*}
\begin{aligned}
& v^{I, 0}(t,x,y) =\eps\tilde {y} \overline{\partial_y v^{I, 0}}+o(\eps),\\
& p^{I, 0}(t,x,y) =\overline{p^{I, 0}}+\eps\tilde {y} \overline{\partial_y p^{I, 0}}+o (\eps).\\
\end{aligned}
\end{equation*} 
Then we insert the above three terms into \eqref{cks} and moreover consider the same order terms of $\eps$. By combining \eqref{pdivc}   
we obtain as follows:

\noindent  
{\bf At the order $\eps^{-1}$:} From the second equation in \eqref{cks}  we get
 \begin{equation*}
 \partial_{\tilde y} p^{B, 0}= 0.
 \end{equation*}
 this together with the assumption that $p^{B, j},j\geq0$ goes to $0$ as $\tilde y\rightarrow+\infty$ implies
 \begin{equation*}\label{np}
 p^{B, 0}\equiv 0.
 \end{equation*} 
\noindent  
{ \bf At the order $\eps^0$:}  
From the second equation in \eqref{cks}, we obtain 
 \begin{equation*}
\partial_{\tilde y} p^{B, 1}= 0.
\end{equation*}
 this together with the assumption that $p^{B, j},j\geq0 $ goes to $0$ as $\tilde y\rightarrow+\infty$ implies
\begin{equation*}\label{np+}
p^{B, 1}\equiv 0.
\end{equation*}
On the other hand, we consider the first equation in \eqref{cks} only with $ O(\eps^0) $ terms to
yield that 
\begin{equation*}
\partial_t^2   u^{B,0}+\partial_t   u^{B,0}+u^{B,0}\partial_x\overline{u^{I, 0}}+(\overline{u^{I, 0}}+u^{B,0})\partial_x\cdot u^{B,0}+(\overline{v^{I, 1}}+u^{B,1}+\tilde{y} \overline{\partial_y v^{I, 0}})\partial_{\tilde{y}}u^{B,0}=\partial_{\tilde{y}}^2u^{B,0}
\end{equation*}
Thus, as a result, we deduce that $  (u^{B, 0},v^{B, 1}) $ satisfies the equations as follows:
\begin{equation*}
 \left\{
\begin{aligned}
&\partial_t^2   u^{B,0}+\partial_t   u^{B,0}+u^{B,0}\partial_x\overline{u^{I, 0}}+(\overline{u^{I, 0}}+u^{B,0})\partial_x\cdot u^{B,0}\\
&\qquad\qquad\qquad\qquad+(\overline{v^{I, 1}}+u^{B,1}+\tilde{y} \overline{\partial_y v^{I, 0}})\partial_{\tilde{y}}u^{B,0}=\partial_{\tilde{y}}^2u^{B,0},\\
&\partial_x\cdot u^{B,0}+\partial_{\tilde{y}}v^{B,1}=0,\\
&u^{B, 0}(t,x,0)=-u^{I, 0}(t,x,0),\quad
v^{B, 1}(t,x,0)=-v^{I, 1}(t,x,0),\quad \lim_{\tilde{y}\rightarrow +\infty}u^{B, 0}(t,x,\tilde{y})=0,\\
&( u^{B, 0},v^{B, 1})|_{ t=0}= (0,0), \quad (\partial_t u^{B, 0},\partial_tv^{B, 1})|_{ t=0}= (0,0).\\
\end{aligned}
\right.
\end{equation*}
Denoting
\begin{equation*}
\begin{aligned}
&u ( t , x , \tilde y  ) =\overline{ u ^{I, 0}}  + u ^{B, 0} ( t , x , \tilde { y } ) ;  \\ 
&v ( t , x ,\tilde y ) =\overline{v^{I, 1}}+u^{B,1}( t , x , \tilde { y } )+\tilde{y} \overline{\partial_y v^{I, 0}}.\\
\end{aligned}
\end{equation*}
which along with the system \eqref{imhdsys} at $ y=0 $ to yield the hyperbolic Prandtl boundary layer equations
\begin{equation}\label{fzk}
\left\{
\begin{aligned}
&\partial_t^2   u+\partial_t   u+u\partial_x\cdot u+v\partial_{\tilde{y}}u+\partial_xP=\partial_{\tilde{y}}^2u,\\
&\partial_x\cdot u+\partial_{\tilde{y}}v=0,\\
&u(t,x,0)=v(t,x,0)=0,\quad \lim_{\tilde{y}\rightarrow +\infty}u(t,x,\tilde{y})=U(t,x),\\
& u|_{ t=0}= \overline{u_{0}(0,x)}, \quad \partial_t u|_{ t=0}= \overline{u_{1}(0,x)},\\
\end{aligned}
\right.
\end{equation}
where 
\begin{equation*}
P(t,x)=\overline{p^{I, 0}}, \quad U(t,x)=\overline{u^{I, 0}}
\end{equation*}
and the known functions $ U(t, x) $ and $ P(t, x) $ satisfy the following law:
\begin{equation*}
\partial_t^2U+\partial_t U+(U\cdot\partial_x)U+\partial_xp=0.
\end{equation*}
For simplicity of notations, we have replaced $ \tilde{y} $ by $ y $  in the boundary layer
system \eqref{mhdsys+++}. 
\begin{remark}
Note that the boundary layer problem \eqref{fzk} can be derived by using the following scale transform:
\begin{equation*}
t=t,x=x,\tilde{y}=\frac{y}{\eps},
\end{equation*}
and
\begin{equation*}
\left\{
\begin{aligned}
&u ^ { \eps } ( t , x , y ) =  u ( t , x , \tilde { y } )  ,\quad v ^ { \eps } ( t , x , y ) =   \eps   v ^ { b } ( t , x , \tilde { y } ) , \\ 
&p ^ { \eps } ( t , x , y ) =  p ( t , x , \tilde { y } ).
\end{aligned}
\right.
\end{equation*}
\end{remark}

\bigskip
\noindent {\bf Acknowledgements.} The work was supported by NSF of China(Nos. 11961160716, 11871054, 12131017)  and  the Natural Science Foundation of Hubei Province No. 2019CFA007.

\end{document}